\documentclass[12pt]{amsart}
\usepackage{latexsym,fancyhdr,amssymb,color,amsmath,amsthm,graphicx,listings,comment}
\usepackage[section]{placeins}
\newtheorem{thm}{Theorem}

\newtheorem{coro}{Corollary}
\let\paragraph\subsection

\title{Arboricity and Acyclic Chromatic Number}

\author{Oliver Knill}
\date{November 5, 2023}
\address{Department of Mathematics \\ Harvard University \\ Cambridge, MA, 02138 }


\begin{document}
\maketitle

\begin{abstract}
A theorem of Hakimi, Mitchem and Schmeichel from 1996 states that the edge 
arboricity arb(G) of a graph is bounded above by the acyclic chromatic number acy(G). 
We can improve this HMS inequality by 1, if acy(G) is even. We review also
results about acyclic chromatic numbers in the context of a Gr\"unbaum conjecture
from 1973. 
\end{abstract}

\section{Summary} 

\paragraph{}
Let $G=(V,E)$ be a {\bf finite simple graph} with vertex set $V$ and edge set $E$. 
The {\bf arboricity} ${\rm arb}(G)$, introduced by Nash-Williams \cite{NashWilliams} 
in 1961, is the minimal number of forests partitioning the edge set $E$.
The {\bf chromatic number} ${\rm chr}(G)$, first considered by Francis Guthrie in 1852
in the context of map coloring, is the maximal number of independent sets 
partitioning the vertex set $V$. The {\bf vertex arboricity} ${\rm ver}(G)$, 
introduced in 1968 \cite{ChartrandKronkWall1968} as {\bf point arboricity}, 
is the maximal number of forests partitioning 
$V$ such that each forest generates itself in $G$. Already \cite{ChartrandKronkWall1968}
point out ${\rm ver}(G) \leq {\rm chr}(G) \leq 2 {\rm ver}(G)$ because every color class is
a forest and every forest is 2-colorable and ${\rm ver}(G) \leq [(1+{\rm deg}(G))/2]$ where
$[x]$ is the least integer not less than $x$ and ${\rm deg}(G)$ is the maximal vertex degree
of $G$. They showed as well ${\rm ver}(G) \leq 3$ for planar graphs, which 
parallels ${\rm arb}(G) \leq 3$ but is unrelated. Determining ${\rm ver}(G)$ is a NP-hard,
while finding ${\rm arb}(G)$ is a polynomial task essentially due to the Nash-Williams formula. 
The {\bf acyclic chromatic number} ${\rm acy}(G)$, 
introduced in 1973 by Gr\"unbaum \cite{Gruenbaum1973},
is the smallest integer for which there is an {\bf acyclic vertex coloring}, meaning that 
all Kempe chains are forests. {\bf Kempe chains} of a vertex-colored graph are sub-graphs of $G$ 
containing only 2 colors. By definition of coloring, Kempe chains are triangle-free but
they can can have cyclic sub-graphs. For an acyclic coloring, all Kempe chains are forests. 

\paragraph{}
All these functionals on networks deal with trees, forests and colors and 
colored trees or forests. The following {\bf foliage inequalities} provides a link.
Besides ${\rm ver}(G) \leq {\rm chr}(G) \leq 2 {\rm ver}(G)$ we have
$$   \frac{{\rm ver}(G)}{2} \leq \frac{{\rm chr}(G)}{2} 
                            \leq {\rm arb}(G) \leq {\rm acy}(G) \; . $$ 
{\bf Tree notions} and {\bf color notions} are interwoven with each other. 
Only the last of the above three inequalities needs some work to be proven. 
We actually will reprove it and improve on it slightly. 

\paragraph{}
The first inequality holds because every vertex coloring is also a vertex forest.
Indeed, each independent set is a forest in which every tree is a seed, a single point.
The second inequality (an exercise in \cite{BM}) 
holds because every forest has chromatic number $1$ (if all trees are seeds) or $2$ (else),
the reason for the later is that every tree can be colored with 2 colors.
The last inequality follows from a result of 
Hakimi, Mitchem and Schmeichel (HMS) from 1996 \cite{HakimiMitchemSchmeichel1996},
who proved that the {\bf vertex star arboricity} is less or equal than 
${\rm acy}(G)$, implying that ${\rm arb}(G) \leq {\rm acy}(G)$. 
{\bf Vertex star arboricity} ${\rm sta}(G)$ sandwiches vertex arboricity because every
tree can be covered with 2 type of stars so that
${\rm sta}(G)/2 \leq {\rm ver}(G) \leq {\rm sta}(G)$. 
(there is also an edge star arboricity that relates in the same way to edge arboricity)
We directly address the HMS inequality and improve it slightly: 

\begin{thm}[Refined HMS inequality] 
${\rm arb}(G) \leq {\rm acy}(G)-1$ if ${\rm acy}(G)$ is even.
Otherwise ${\rm arb}(G) \leq {\rm acy}(G)$.
\end{thm}

\paragraph{}
\begin{proof}
Assume that the acyclic chromatic number is $c$. 
This means that there are $c(c-1)/2$ different Kempe chains and 
that each of these chains is either a forest or empty. 
If $c$ is even, then we can bundle the Kempe colors into $c/2$ disjoint pairs.
Bundling the Kempe chains as such gives now $(c-1)$ color types
and so $c(c-1)/2/(c/2) = c-1$ forests.
In the case when $c$ is odd, we can only form $(c-1)/2$ parts and need
to leave one color alone. We count then $[c (c-1)/2]/((c-1)/2) = c$ forests. \\
\end{proof}

\paragraph{}
We had already made use of this in \cite{ThreeTreeTheorem} in the case $c=4$, 
where we have $6$ different Kempe chains $AB,AC,AD,BC,BD,CD$, lead to the 3 type of forests,
the union of $AB,CD$ Kempe chains, the union of $AC,BD$ Kempe chains and the union
of $AD,BC$ Kempe chains. Since the vertex sets of the AC and BD Kempe chains are disjoint,
the union of the AC forest and BD forest remains a forest. 

\section{Remarks}

\paragraph{}
For all 1-manifolds, we have  ${\rm arb}(G)=2, {\rm chr}(G) \in \{ 2,3\}$ and ${\rm acy}(G)=3$.
For all 2-spheres, ${\rm arb}(G)=3, {\rm chr}(G) \in \{3,4\}$ by the {\bf 4-color theorem}
with ${\rm chr}(G)=3$ characterized by {\bf Eulerian 2-spheres} (2 spheres for which every
vertex degree is even, something which happens for example if $G$ is a Barycentric refinement), 
and ${\rm acy}(G) \in \{4,5\}$, where the case $5$ only happens for prisms. Prisms are 
very special 2-spheres for example because they are the only non-prime spheres in the Zykov
monoid. 

\paragraph{}
It had been a conjecture of Gr\"unbaum proven by Borodin in 1979 \cite{Borodin1979}, that
for planar graphs ${\rm acy}(G) \leq 5$. 
This came after ${\rm acy}(G) \leq 7$ \cite{AlbertsonBerman1977}.
We proved that ${\rm acy}(G) \leq 4$
for planar graphs, unless we have a prismatic graph. This improves on Gr\"unbaum
who by the way already pointed out that graphs like the octahedron have acyclic chromatic
number $5$. 

\paragraph{}
For other 2-manifold types, we know of cases with chromatic number
${\rm chr}(G) \in \{3,4,5\}$. A conjecture of 
Albertson and Stromquist states that for 2-manifolds (graphs for which every unit sphere
is a cyclic graph with 4 or more vertices), no larger chromatic number than 5 is possible. 
Still for $2$-manifolds, we have ${\rm arb}(G)=3$ for $3$-spheres and ${\rm arb}(G)=4$ for all
other topological types. We have so far only seen 
${\rm acy}(G) \in \{4,5\}$ for 2-manifolds. Even if the Albertson-Stromquist conjecture
should hold and ${\rm chr}(G) \leq 5$ for all 2-manifolds, it could still 
be that the acyclic chromatic number could be bigger than 5 for some manifolds. 

\paragraph{}
Peter Tait proved that the edge arboricity ${\rm arb}(G)$ of the dual $G^*$ of a 2-sphere
$G$ is less or equal than 3. This follows directly from the 4-color theorem:
with a vertex 4-coloring of $G$, one
has immediately an {\bf edge coloring} of $G^*$ with 3 colors. Given a 4 coloring $(0,a,b,c)$
of the vertices of $G$, one can identify the elements $0,a,b,c$ as  as elements of the Klein
4-group and define the edge coloring $f( (a,b) ) = a+b$. This is not an edge coloring of $G$
but an edge coloring of $G^*$. Conversely, if
a 3-coloring $(a,b,c)$ of the edge set of $G^*$ (which agrees with the edge set of $G$) 
is given, we necessarily have the colors $a,b,c$ in each triangle of $G$. 
The property $a+b+c=0$ can be seen as a {\bf zero curl condition}, implying
that this ``vector field" comes from a gradient field so that $f(a,b) = b-a$ which is $b+a$
in the Klein 4-group. This is explained for example in \cite{Aigner2015}. 
Note however that the edge coloring number of $G^*$ is larger than the arboricity of $G^*$
which is $2$ because the vertex degree of $G^*$ is constant $3$. 

\paragraph{}
The arboricity ${\rm arb}(G)$ of a graph $G$ is a measure for the network's density.
It is the minimal number of forests that partition the graph and so is a packing number.
By the {\bf Nash-Williams theorem} \cite{NashWilliams,NashWilliams1964}, it is
is the smallest integer $k$ larger or equal than the {\bf Nash-Williams bound}
$W(G) = {\rm max}_{H \subset G} |E_H|/(|V_H|-1)$ over all induced 
sub-graphs $H=(V_H,E_H)$ of $(V,E)$. Unlike the arboricity which is the edge arboricity,
the vertex arboricity does not have such a formula. Indeed, determining vertex
arboricity is NP hard, while determining edge arboricity is of polynomial difficulty. 
For more results on vertex arboricity, see \cite{JensenToft}. 
By the way, also the problem of acyclic coloring is NP complete \cite{ColemanCai}.

\paragraph{}
The {\bf empty graph} $0$ is the $(-1)$-sphere. The {\bf 1-point graph} $1$ is defined to be
{\bf contractible}. A {\bf d-sphere} is a finite simple graph for which the unit sphere is
a $(d-1)$-sphere and the removal of some vertex $v$ produces a contractible graph
$G-v$. A graph is {\bf contractible} if there is a vertex with contractible unit
sphere $S(v)$ such that also $G-v$ is contractible. A {\bf d-manifold} is a finite simple graph 
for which every unit sphere $S(v)$ is a $(d-1)$-sphere. 
The smallest arboricity which a d-manifold can achieve is $d+1$, obtained 
by cross polytopes.  For $d \geq 3$, the arboricity can be arbitrarily large for any 
topological type (already pointed out in \cite{ArboricityManifolds}): 

\begin{coro}
a) For any d-manifold type with $d>2$ there are discrete manifolds for which 
the acyclic chromatic number is arbitrarily large. \\
b) The smallest arboricity which can be achieved for d-manifolds is $d+1$. 
We do not know whether the lower bound $d+1$ can be reached for any non-sphere. 
\end{coro}

\begin{proof}
a) If we want to reach a target arboricity $a$, first make Barycentric refinements until 
for some edge $e=(a,b)$ the $(d-2)$-sphere $S(a) \cap S(b)$
has $a-1$ or more vertices. Now, every edge refinement of $e$ adds one vertex and
at least $a$ edge. Repeat this until $E/(V-1)$ is larger or equal than $a-1$.
But this means by the Nash-Williams theorem that the arboricity is larger or 
equal than $a$. \\
b) The Euler handshake formula shows $2E=\sum_{v \in V} {\rm deg}(v)$. 
The smallest $(d-1)$-sphere has $2(d-1)+2=2d$ vertices, so that ${\rm deg}(v) \geq 2d$. 
This shows $E/V \geq d$ and so $E/(V-1) \geq d+1$. 
\end{proof}

\paragraph{}
On every {\bf Erd\"os-R\'enyi probability space} $E(n,p)$, the expected value 
of the Nash-Williams functional is for $n>1$ equal to 
${\rm E}|_{n,p}[W(G)] = p n/2$ simply because the 
Nash-Williams ratio $W(H) = |E_H|/(|V_H|-1)$ for any sub-graph 
has the expectation $p n (n-1)/(n-1) = p n/2$. We do not know 
what the expectation of the arboricity is although. We know the 
expectation of Euler characteristic or inductive dimension but
the expectation of arboricity or chromatic number functionals on
$E(n,p)$ appears to be difficult to establish. 

\paragraph{}
Arboricity is related to various other packing or covering problems on graphs. 
The {\bf star arboricity} \cite{AlgorAlon1989}, the {\bf linear arboricity}
and the {\bf caterpillar arboricity} for example fit in as
${\rm star}(G) \geq {\rm cater}(G) \geq {\rm arb}(G) \geq {\rm star}(G)/2$, where
the last inequality follows from the fact that every forest can be colored with 
2 stars.  The arboricity ${\rm arb}(G) \geq {\rm cat}(G)$ is also an upper bound 
for the {\bf Lusternik-Schnirelmann category} ${\rm cat}(G)$ of the graph, 
which is the number of contractible graphs which are needed to cover the network. 
Since the {\bf augmented cup length}$ {\rm cup}(G)+1$  of the graph is a 
lower bound for the category, this cohomological notion is also a lower 
bound for the arboricity ${\rm arb}(G)$. 

\bibliographystyle{plain}

\end{document}